\documentclass[reqno]{amsart}
\usepackage{amsmath}
\usepackage{amsthm}
\usepackage{amssymb}
\usepackage{epsfig,eepic,epic}
\usepackage{tikz}

\newcommand{\norow}{\ensuremath{\underline{\lambda}}} 
\newcommand{\F}[1]{\ensuremath{F(#1)}} 
\newcommand{\St}{\ensuremath{\sigma}}

\theoremstyle{plain}
\newtheorem{thm}{Theorem}[section]
\newtheorem{cor}[thm]{Corollary}

\newtheorem{lem}[thm]{Lemma}
\newtheorem{prop}[thm]{Proposition}
\theoremstyle{definition}
\newtheorem{defn}[thm]{Definition}

\theoremstyle{remark}
\newtheorem{rem}[thm]{Remark}

\newcommand{\B}{\ensuremath{\beta}}

\begin{document}

\title{Boolean complexes for Ferrers graphs}

\author[A.~Claesson]{Anders Claesson}
\author[S.~Kitaev]{Sergey Kitaev}
\address{School of Computer Science, Reykjav\'ik University, Iceland}
\email{anders@ru.is\\ sergey@ru.is}

\author[K.~Ragnarsson]{K\'ari Ragnarsson}
\author[B.~E.~Tenner]{Bridget Eileen Tenner}
\address{Department of Mathematical Sciences, DePaul University, Chicago, Illinois, USA}
\email{kragnars@math.depaul.edu\\ bridget@math.depaul.edu}

\thanks{Claesson and Kitaev were supported by the Icelandic Research Fund, grant no. 090038011.}

\keywords{Ferrers graph, Ferrers shape, partition,
Legendre-Stirling numbers, Genocchi numbers of the second kind,
Stirling number of the second kind, (complete) bipartite graph,
homotopy type, boolean complex, boolean number, complexity}

\subjclass[2000]{Primary 05A15; Secondary 55P15, 05C99, 05A19}

\begin{abstract}
In this paper we provide an explicit formula for calculating the boolean 
number of a Ferrers graph. By previous work of the last two authors, this 
determines the homotopy type of the boolean complex of the graph.
Specializing to staircase shapes, we show that the boolean numbers 
of the associated Ferrers graphs are the Genocchi numbers of the second
kind, and obtain a relation between the Legendre-Stirling numbers and the 
Genocchi numbers of the second kind.  In another application, we compute 
the boolean number of a complete bipartite graph, corresponding to a rectangular 
Ferrers shape, which is expressed in terms of the Stirling numbers of the 
second kind. Finally, we analyze the complexity of calculating the boolean 
number of a Ferrers graph using these results and show that it is a significant
improvement over calculating by edge recursion.
\end{abstract}

\maketitle
\thispagestyle{empty}

\section{Introduction}

Ferrers shapes, or Young shapes or partitions, are classical
combinatorial objects arising in a variety of contexts including
Schubert varieties, symmetric functions, hypergeometric series,
permutation statistics, quantum mechanical operators, and inverse rook
problems (see references in~\cite{CN}).  To such an object, one can
relate a bipartite graph known as a \emph{Ferrers graph}, as
introduced in \cite{EW}.

Let $\lambda = (\lambda_1, \ldots, \lambda_r)$ be a partition, where
$\lambda_1 \ge \cdots \ge \lambda_r \ge 0$.  The
associated bipartite Ferrers graph has vertices $\{x_1, \ldots, x_r\} \sqcup \{y_1, \ldots, y_{\lambda_1}\}$, and edges $\big\{\{x_i, y_j\} : \lambda_i \ge j\big\}$.  In
    particular, vertex $x_i$ has degree $\lambda_i$.
    A Ferrers graph and its
    associated Ferrers shape are depicted in Figure~\ref{fig1}.  Note that if $\lambda_i=0$ then $x_i$ is an isolated vertex in this graph.

\setlength{\unitlength}{4mm}
\begin{figure}[h]
\begin{center}
\begin{tikzpicture}
\foreach \x in {1,2,3} {\draw (1,1) -- (\x,2);}
\foreach \x in {1,2,3} {\draw (2,1) -- (\x,2);}
\foreach \x in {1,2} {\draw (3,1) -- (\x,2);}
\foreach \x in {1,2} {\draw (4,1) -- (\x,2);}
\foreach \x in {1,2,3,4} {\fill[black] (\x,1) circle (2pt);}
\foreach \x in {1,2,3} {\fill[black] (\x,2) circle (2pt);}
\foreach \i in {1,2,3,4} {\draw (\i,1) node[below] {$y_{\i}$};}
\foreach \i in {1,2,3} {\draw (\i,2) node[above] {$x_{\i}$};}
\end{tikzpicture}
\hspace{.5in}
\begin{tikzpicture}[scale=.6]
\draw (1,1) -- (3,1);
\draw (1,2) -- (5,2);
\draw (1,3) -- (5,3);
\draw (1,4) -- (5,4);
\draw (1,1) -- (1,4);
\draw (2,1) -- (2,4);
\draw (3,1) -- (3,4);
\draw (4,2) -- (4,4);
\draw (5,2) -- (5,4);
\end{tikzpicture}
\end{center}
\caption{A Ferrers graph and its associated Ferrers shape $\lambda=(4,4,2)$.}\label{fig1}
\end{figure}
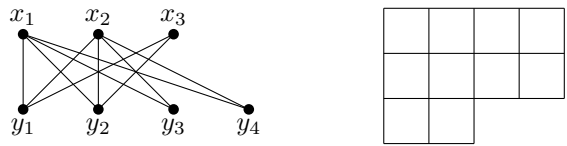

A selection of enumerative properties of Ferrers graphs are studied
in~\cite{EW}, where the graphs are introduced. In particular,
expressions for the number of spanning trees, the number of
Hamiltonian paths, the chromatic polynomial, and the chromatic
symmetric function are given.

In~\cite{EK}, the authors find the number of independent sets of a
Ferrers graph, and give relations between the set of independent sets
of a Ferrers graph and certain combinatorial objects. Moreover, it is
shown in~\cite{EK} that the simplicial complex related to the set of
independent sets of a non-rectangular Ferrers graph is
simple-homotopic to a point, whereas it is simple-homotopic to two
points in the case of a rectangular Ferrers graph.

Monomial and toric ideals associated to Ferrers graphs are studied
in~\cite{CN}. In particular, it is shown that the edge ideal of a
Ferrers graph, called the Ferrers ideal, has a $2$-linear minimal free
resolution.  That is, it defines a small subscheme, which is proved to
characterize Ferrers graphs among bipartite graphs.

In this paper, we study the homotopy type of the boolean complexes of
Ferrers graphs. Roughly speaking, the boolean complex of a graph $G$
is the complex of words on the vertex set of $G$, without repeated
letters, where two letters commute if the corresponding vertices are
not connected by an edge in $G$. Boolean complexes were introduced in
\cite{ragnarsson tenner}, where it is shown that the boolean complex
of a graph $G$ on $n$ vertices always has the homotopy type of a wedge
of spheres of dimension $n-1$. The homotopy type is therefore
determined by the number of spheres in the wedge sum, which we denote
$\B(G)$, and refer to as the \emph{boolean number} of $G$.

The paper is organized as follows. In Section~\ref{sec:previous stuff}
we recall the definitions and results on Boolean complexes from
\cite{ragnarsson tenner} that will be needed in this paper, as well as 
more general functions related to the boolean number and some well-known 
sequences that arise in the course of this article. In
Section~\ref{general} we provide a recursion for calculating the
boolean number of an arbitrary Ferrers graph in terms of certain
truncated shapes (Theorem~\ref{thm:Recursion}).  This formula is used
in Section~\ref{formula-general} to derive an identity
(Theorem~\ref{thm:PowerSums}) for the boolean number in terms of
certain recursively defined coefficients $c_{\lambda}(n,j)$. 
More precisely, if
$F$ is the Ferrers graph associated to a Ferrers shape $\lambda =
(\lambda_1,\ldots,\lambda_r)$, then
\[ \B(F) = \sum_{j=0}^r c_\lambda(r,j) \cdot j^{\lambda_r}. \]

We restrict to Ferrers graphs for staircase shapes in
Section~\ref{staircases}, obtaining a connection to the
Legendre-Stirling numbers.  This yields an explicit double sum formula for $\B(F)$ in
this case (Corollary~\ref{cor:staircase}).  Intriguingly, the boolean numbers of staircase shapes are the Genocchi numbers
of the second kind. As a corollary, we find a
relationship between the Legendre-Stirling numbers $\{d(r,j)\}$ and
the Genocchi numbers of the second kind $\{g(r)\}$ that seems to be
new in the literature:
$$g(r) = \sum\limits_{j=1}^r (-1)^{r+j}(j!)^2 \cdot d(r,j).
$$

Applying Theorem~\ref{thm:Recursion} to complete bipartite graphs
$K_{r,k}$ in Section~\ref{complete}, we express $\B(K_{r,k})$ in terms of Stirling numbers of the second kind:
$$\B(K_{r,k}) = \sum_{j=1}^{r}(-1)^{r-j}j! \begin{Bmatrix} r+1 \\ j+1 \end{Bmatrix}j^k.$$

Finally, in Section~\ref{sec:complexity}, we analyze the complexity of 
computing the boolean number of a Ferrers graph using the result of 
Theorem~\ref{thm:PowerSums}, and we show that this is a significant
improvement over calculating the boolean number using the edge-recursion 
from \cite{ragnarsson tenner}.

\section{Background material}\label{sec:previous stuff}

\subsection{Boolean complexes.}\label{some-def}
The motivating object in this paper is the boolean complex of a
graph, as defined in \cite{ragnarsson tenner}.  We now recall the
basic definitions relating to boolean complexes, and
the main result about their homotopy type (Theorem~\ref{key}). The
reader is referred to \cite{ragnarsson tenner} for a thorough
discussion.

\begin{defn}\label{defn:boolean complex}
  Let $G$ be a finite simple graph with vertex set $V$. Construct a
  simplicial poset $\mathbb{B}(G)$ whose elements are equivalence
  classes of strings of distinct elements of $V$, where two strings
  are equivalent if one can be transformed into the other by commuting
  elements that are non-adjacent in $G$. The partial order on
  $\mathbb{B}(G)$ is induced by substring inclusion.  To the poset
  $\mathbb{B}(G)$ we associate the regular cell complex $\Delta(G)$,
  called the \emph{boolean complex} of $G$.  The geometric realization
  of this complex is denoted $|\Delta(G)|$.
\end{defn}

As discussed in \cite{ragnarsson tenner} and reiterated in Theorem~\ref{key} below, 
this $|\Delta(G)|$ is homotopy equivalent to the wedge of $\B(G)$ spheres, and the
boolean number $\B(G)$ of a graph can be calculated recursively using three edge
operations. 

\begin{defn}\label{operations} 
Let $G$ be a (multi-)graph and $e$ an edge in $G$.
\begin{itemize}
\item \emph{Deletion:} $G-e$ is the graph obtained by deleting the edge $e$.
\item \emph{Contraction:} $G\downarrow e$ is the graph obtained by contracting the edge $e$.
\item \emph{Simple contraction:} When $G$ is a simple graph, $G/e$ is the graph obtained from $G\downarrow e$ by removing all loops and redundant edges.
\item \emph{Extraction:} $G-[e]$ is the graph obtained by removing the edge $e$ and its incident vertices.
\end{itemize}
\end{defn}

\begin{defn} 
  For a finite graph $G$, let $|G|$ denote the number of vertices in $G$.
  Also, for $n\geq 1$, let $\delta_n$ be the graph consisting of $n$
  disjoint points.
\end{defn}

\begin{defn}
  For integers $b, r \ge 0$, the notation $b\cdot S^r$ is used to
  indicate a wedge sum of $b$ spheres of dimension $r$.  In
  particular, $0\cdot S^r$ is a single point.
\end{defn}

We can now state the main theorem on the homotopy type of a boolean
complex, as well as another useful result of \cite{ragnarsson tenner}. 
The symbol $\simeq$ is used in the statement to denote
homotopy equivalence.

\begin{thm}[{\cite[Theorem 3.4]{ragnarsson tenner}}]\label{key}
  For every finite simple graph $G$,
  there is an integer $\B(G)$ so that $|\Delta(G)|\simeq \B(G)\cdot
  S^{|G|-1}$.  Moreover, the values $\B(G)$ can be computed recursively using the equations
  \begin{align*}
    \B(G) &= \B(G-e)+\B(G/e)+\B(G-[e])\; 
    \text{ if $e$ is an edge in $G$},\\
    \B(\delta_n) &= 0, \text{ and}\\
    \B(\emptyset) &= 1.
\end{align*}
\end{thm}

\begin{prop}[{\cite[Corollary 7.2]{ragnarsson tenner}}] \label{contractible}
  A finite simple graph $G$ satisfies $\B(G)=0$ if and only if $G$ has
  an isolated vertex.
\end{prop}

\subsection{Connections between the boolean number and other functions}

The function $\B$ is related to several functions that 
have been studied previously; namely the universal edge elimination polynomial \cite{agm}, the bivariate chromatic polynomial  \cite{dpt}, and the rank generating function.  We highlight these connections primarily for the sake of context, but also 
because they enable the complexity analysis of Section~\ref{sec:complexity}. The reader is refered to \cite{agm} and \cite{dpt} for more information about these functions.

The universal edge elimination polynomial $\xi$ was introduced in \cite{agm}. It is defined on multi-graphs
$G$ and is determined by the recursive definition 
\begin{eqnarray*}
  \xi(G,x,y,z) &= &\xi(G-e,x,y,z) + y\xi(G\downarrow e,x,y,z) + z\xi(G - [e],x,y,z), \\
  \xi(G_1 \sqcup G_2, x,y,z) &= &\xi(G_1, x,y,z) \cdot \xi(G_2, x,y,z), \\
  \xi(\delta_1) &= &x, \text{ and} \\
  \xi(\emptyset) &= &1.
\end{eqnarray*}
The ``universal'' property of $\xi$ is that any polynomial defined on multi-graphs 
that satisfies a linear edge-recurrence relation is an evaluation of $\xi$. To apply 
this property to $\B$, which satisfies a linear edge recurrence but is only defined
on simple graphs, we use the following lemma, which follows immediately from the
recurrence for $\xi$.

\begin{lem}\label{lem: strip edges in xi}
Fix a graph $G$ and an edge $e$.  If more than one edge connects the endpoints of $e$, 
then $\xi(G,x,-1,z) = \xi(G-e,x,-1,z)$. 
\end{lem}

When $G$ is simple, the graphs $G-e$ and $G-[e]$ are simple as well.  However, $G\downarrow e$ may have multiple edges (although not
loops). Using Lemma \ref{lem: strip edges in xi} we deduce that, for a simple graph $G$,
 \[ \xi(G,x,-1,z) = \xi(G-e,x,-1,z) - \xi(G/ e,x,-1,z) + z\xi(G - [e],x,-1,z).\]
In particular, we obtain the following characterization of $\B$.

\begin{prop} \label{prop:bool as xi} If $G$ is a simple graph, then $\B(G) = (-1)^{|G|}\xi(G,0,-1,1)$.
\end{prop}
\begin{proof}
The two invariants satisfy the same recursion with the same initial conditions.
\end{proof}

Notice that the graph invariant $\xi(G,0,1,1)$ satisfies the recurrence
$$\xi(G,0,1,1) = \xi(G-e,0,1,1) + \xi(G\downarrow e,0,1,1) + \xi(G - [e],0,1,1).$$
However it is not true that $\B(G) = \xi(G,0,1,1)$, as $\xi(G\downarrow e,0,1,1) \neq \xi(G / e,0,1,1)$
in general.

For a graph $G$ and $x,y \geq 0$, the bivariate chromatic polynomial $P(G,x,y)$ counts the number of 
colorings of $G$ using $y$ proper colors and $x-y$ improper colors. In such
a coloring, two vertices colored by the same proper color may not be connected
by an edge, but one may have the same improper color on both endpoints of 
an edge.  Using a linear edge recurrence, it is shown in \cite{agm} that $P(G,x,y) = \xi(G,x,-1,x-y).$
Proposition \ref{prop:bool as xi} then implies the following result.

\begin{cor} \label{cor:bool as bichrom}
If $G$ is a simple graph, then $\B(G) = (-1)^{|G|}P(G,0,-1)$.
\end{cor}

Also, if $R_P(t)$ is the rank generating function of a ranked poset $P$, then one easily obtains
\[ \B(G) = (-1)^{|G|} \cdot R_{\mathbb{B}(G)}(-1) \]
by comparing each side of the equation to the Euler characteristic
of $\mathbb{B}(G)$.

\subsection{Sequences appearing in this paper}

The {\em Legendre-Stirling numbers} are defined in~\cite{ELW} as
\begin{equation} \label{eq:L-S}
  d(i, j) 
  = \sum_{\ell=1}^{j} 
  \frac{(-1)^{\ell+j}(2\ell+1)(\ell^2+\ell)^i}{(\ell+j+1)!(j-\ell)!}. 
\end{equation}
These are sequence A071951 of \cite{oeis}.

The {\em Genocchi numbers of the second kind} (also known as the {\em
  median Genocchi numbers}) have several definitions (see~\cite[A005439]{oeis} and
references therein). One interpretation is that they count permutations $a_1a_2\cdots a_{2n+1} \in S_{2n+1}$ such that $a_i>i$ if $i$ is odd and $i<n$, and $a_i \leq i$ if $i$ is even.

The well-known {\em Stirling numbers of the second kind} count
partitions of an $n$ element set into $k$ nonempty
blocks. They form sequence A008277 of \cite{oeis}.

\section{Recursion for a general Ferrers shape}\label{general}

As mentioned previously, we make the convention that a Ferrers shape
has a specified number of rows, even if some of these rows are
empty. Such a Ferrers shape corresponds to a partition into a
prescribed number of parts, where some parts are allowed to be zero.

\begin{defn} 
  For a Ferrers shape $\lambda$, the Ferrers graph associated to
  $\lambda$ is denoted $\F{\lambda}$.  When no confusion will arise,
  the notation $\B(\lambda)$ will be taken to mean $\B(\F{\lambda})$.
  If the shape $\lambda$ has $r$ rows and $\lambda_1$ columns, then
  the vertices of $\F{\lambda}$ will be denoted $\{x_1, \ldots, x_r\} \sqcup \{y_1, \ldots, y_{\lambda_1}\}$,
  and there is an edge $\{x_i, y_j\}$ if and only if $\lambda_i \ge j$.
\end{defn}

If $\lambda$ has a row of length zero, that is, if some $\lambda_i$
equals $0$, then the corresponding vertex $x_i$ has no incident
edges. Consequently, the boolean number of such a graph is $0$, by
Proposition \ref{contractible}.

The aim of this section is to obtain a recursive formula for the
boolean number of a Ferrers shape, based on the length of its bottom
row. First we define the shapes appearing in the recursion.

\begin{defn}
  For a Ferrers shape $\lambda = (\lambda_1, \ldots, \lambda_r)$ with
  $r>1$ rows, set $\norow$ to be the shape $(\lambda_1, \ldots,
  \lambda_{r-1})$, obtained by deleting the bottom
  row from $\lambda$.
\end{defn}

\begin{defn}
For a Ferrers shape $\lambda = (\lambda_1, \ldots, \lambda_r)$, and an integer $t \ge - \lambda_r$, define the shape
$$\lambda[t] = (\lambda_1 + t, \lambda_2 + t, \ldots, \lambda_r + t).$$
\end{defn}

The shape $\lambda[t]$ is obtained from $\lambda$ by appending $t$
columns of length $r$ to the left side of the shape $\lambda$. If $t <
0$, then these columns are actually deleted from $\lambda$. When $t =
-\lambda_r$, this means that all the boxes in the bottom row of
$\lambda$ are deleted, so the bottom row of $\lambda[-\lambda_r]$ is
empty. Furthermore, for all $i$ such that $\lambda_i = \lambda_r$, the
$i$-th row of the shape $\lambda[-\lambda_r]$ is empty.

\begin{thm} \label{thm:Recursion}
  The boolean number of the Ferrers graph associated to the
  shape $\lambda = (\lambda_1, \ldots, \lambda_r)$, can be computed
  recursively according to the formula
\begin{equation}\label{eqn:recursive formula}
  \B(\lambda) =
  \begin{cases}
    1 & \text{ if } r = 1;\\
    \lambda_r \cdot \B(\norow) + \sum\limits_{\ell=1}^{\lambda_r} \binom{\lambda_r+1}{\ell+1} \cdot \B(\norow[-\ell]) & \text{ if } r>1.
  \end{cases}
\end{equation}
\end{thm}

\begin{proof}
  If $r = 1$, then the graph is a star, and the boolean number of this
  graph is $1$ (see \cite{ragnarsson tenner}).
  The remainder of the theorem is proved by induction on $\lambda_r$.
  
  Assume that $r>1$.  Let the
  vertex $x_r \in \F{\lambda}$ correspond to the last ($r$-th) row of $\lambda$.
  Thus the degree of $x_r$ is $\lambda_r$.  We apply the edge-recursion of
  Theorem~\ref{key} at the edge $\{x_r, y_{\lambda_r}\}$, which
  corresponds to the rightmost box of the bottom row in $\lambda$. The recursion
  involves three operations: deletion, extraction, and simple contraction. The first
  two of these operations translate easily into the language of
  Ferrers graphs. More precisely, deleting $\{x_r,
    y_{\lambda_r}\}$ corresponds to deleting the $\lambda_r$-th box
  from the $r$-th row of $\lambda$, which means subtracting $1$ from
  the last part of the partition $\lambda$.  Note that we still
  require that this shape have $r$ rows, although the bottom row will
  be empty if $\lambda_r = 1$. Likewise, extracting $\{x_r,
    y_{\lambda_r}\}$ corresponds to
  deleting the entire $r$-th row and $\lambda_r$-th column from
  $\lambda$, which gives the shape $\norow[-1]$.

  Thus it remains to understand what happens when $\{x_r, y_{\lambda_r}\}$ is contracted.  Unfortunately, if
  $\lambda_r > 1$, the resulting graph is no longer bipartite, and so
  does not correspond to a Ferrers shape.  However, if $\lambda_r =
  1$, then contracting $\{x_r, y_1\}$ yields the graph
    $\norow$.  In this case, when $\lambda_r = 1$, the graph obtained
    after deleting $\{x_r, y_1\}$ has an isolated vertex,
      which has boolean number $0$.  Thus, if $\lambda_r = 1$, then
      $\B(\lambda)$ equals $\B(\norow) + \B(\norow[-1])$,
      which proves equation~\eqref{eqn:recursive formula} in the base case.

  Now suppose that the result has been proved when the last row of the
  shape has length less than $\lambda_r$.  Then deleting the edge
  $\{x_r, y_{\lambda_r}\}$ contributes
  $$\B\big((\lambda_1, \ldots, \lambda_{r-1}, \lambda_r-1)\big) =
  (\lambda_r-1) \cdot \B(\norow) + 
  \sum\limits_{\ell=2}^{\lambda_r} \binom{\lambda_r}{\ell} \cdot \B(\norow[-(\ell-1)]),$$
  to $\B(\lambda)$, while extracting $\{x_r,
    y_{\lambda_r}\}$ contributes $\B(\norow[-1])$.
  Combining these values gives the sum
  \begin{equation}\label{eqn:partial sum}
    (\lambda_r-1) \cdot \B(\norow) + \left(\binom{\lambda_r}{2} + 1 \right) \cdot
    \B(\norow[-1]) + \sum\limits_{\ell=3}^{\lambda_r} \binom{\lambda_r}{\ell} \cdot
    \B(\norow[-(\ell-1)]).
  \end{equation}

  For a Ferrers shape $\mu$, let $F'(\mu)$ be the (likely
  non-bipartite) graph obtained from $F(\mu)$ by contracting the edge
  corresponding to the rightmost box in the bottom row of $\mu$.  With $F'(\lambda)$ defined in this way, the
  boolean number $\B(\lambda)$ is equal to the sum of $\B(F'(\lambda))$ and
  the expression in \eqref{eqn:partial sum}.  We prove by induction on
  $\lambda_r \ge 2$ that $\B(F'(\lambda))$ equals
  \begin{equation}\label{eqn:contraction}
    \B(\norow) + (\lambda_r-1) \cdot \B(\norow[-1]) + \sum\limits_{\ell=2}^{\lambda_r} \binom{\lambda_r}{\ell} \cdot \B(\norow[-\ell]).
  \end{equation}

  This is straightforward to show
  if $\lambda_r=2$, because the first term corresponds to deleting the
  edge between $y_1$ and $y_2$ in $F'(\lambda)$, the second term
  corresponds to contracting this edge, and the last term represents
  extracting the edge.  Now suppose
  inductively that the equality holds for all shapes whose last rows
  have fewer than $\lambda_r$ boxes.

  Deleting the edge $\{y_1,y_{\lambda_r}\}$ from the graph
  $F'(\lambda)$ yields the graph
  $$F'\big((\lambda_1, \ldots, \lambda_{r-1}, \lambda_r-1)\big).$$
  Likewise, extracting the edge $\{y_1, y_{\lambda_r}\}$ gives the graph $\F{\norow[-2]}$.  Finally,
  simply contracting the edge $\{y_1, y_{\lambda_r}\}$ yields the graph
  $F'(\lambda[-1])$.  Hence
  $$\B\big(F'(\lambda)\big) = \B\big(F'((\lambda_1, \ldots,
  \lambda_{r-1}, \lambda_r-1))\big) + \B(\norow[-2]) +
  \B\big(F'(\lambda[-1])\big),$$ 
  and $\B(F'(\lambda))$ equals the
  expression in~\eqref{eqn:contraction} by the inductive hypothesis
  and a binomial identity.

  Finally, we combine the expressions in~\eqref{eqn:partial sum}
  and~\eqref{eqn:contraction} to complete the proof.
\end{proof}

\section{Formula for general Ferrers shape}\label{formula-general}

In this section we obtain an explicit formula for the boolean number
of a Ferrers graph.

\begin{defn}
For a Ferrers shape $\lambda = (\lambda_1,\ldots,\lambda_r)$ with
$r$ rows, define the numbers $c_\lambda(i,j)$, where $1 \leq i \leq
r$ and $j \in \mathbb{Z}$, recursively by
$$ 
c_\lambda(1,j) =
\begin{cases}
  -1 & \text{~if~} j=0,\\
  1 & \text{~if~} j=1,\\
  0 & \text{~if~} j\not\in \{0,1\},
\end{cases}
$$
and, for $1 < i \leq r$,
$$ c_\lambda(i,j) = j(j-1)^{(\lambda_{i-1}-\lambda_{i})} \cdot
c_\lambda(i-1,j-1) - (j+1)j^{(\lambda_{i-1}-\lambda_{i})}  \cdot c_\lambda(i-1,j).
$$
\end{defn}

Here we use the convention $0^0 =1$. Note that it follows directly
from the definition that $c_\lambda(i,j)$ only takes nonzero values
when $0\leq j \leq i$, and so the values $c_\lambda(i,j)$ can be
calculated by means of a triangular array. The zero values $c_\lambda(i,j)$
for $j <0$ or $j>i$ play no role in the paper, but are included in
the definition so that we avoid exceptions in the recursive
definition.

As an example, the triangle used to calculate $c_{\lambda}(i,j)$ for
the Ferrers shape $\lambda = (7,7,7,6,4,4,2)$ is given in Table
\ref{tbl:Triangle}. This triangle exhibits three interesting
phenomena, each of which can be shown to hold for any Ferrers
shape. First the entries in the leftmost column are zero after the third row.  (In general one has $c_{\lambda}(i,0)=0$ when $\lambda_i < \lambda_1$.) Second, when one disregards the leftmost
column, adjacent entries in the triangle have alternating
signs. Third, the entries in each row sum to zero.

\begin{table}[htbp]
  \begin{center}\small{ 
      $\begin{array}{c|rrrrrrrr}i 
        & j=0 & 1 & 2 & 3 & 4 & 5 & 6 & 7\\
        \hline1 
        &  -1 & 1    \\
        2 &  1  & -3  & 2       \\
        3 &  -1 &  7  & -12   & 6      \\
        4 & 0   & -14 & 86    & -144  &  72 \\
        5 &  0  & 28  & -1060 & 6216  & -10944 & 5760 \\
        6 & 0   & -56 & 3236  &-28044 & 79584  & -89280 & 34560\\
        7 & 0   & 112 & -38944 & 1048416 & -7376304 & 19758720 & -22101120 & 8709120\end{array}$}
  \end{center}
  \caption{The triangle calculating $c_{\lambda}(i,j)$ 
    for $\lambda =(7,7,7,6,4,4,2)$.} \label{tbl:Triangle}
\end{table}

\begin{lem} \label{lem:Samec}
  Let  $\lambda = (\lambda_1,\ldots,\lambda_r)$ be a Ferrers shape with
  $r$ rows.
  \begin{enumerate}
  \item[(a)] $c_{\norow}(i,j) = c_{\lambda}(i,j),$ for all integers $i$ and $j$ with $1 \leq i\leq r-1$.
  \item[(b)] If $t \geq -\lambda_r$, then $c_{\lambda[t]}(i,j) = c_{\lambda}(i,j)$ for all integers $i$ and $j$ with $1 \leq i\leq r$.
  \end{enumerate}
\end{lem}

\begin{proof}
From the recursive definition of the numbers $c_\lambda(i,j)$, we see that they depend only on the differences $\lambda_{\ell-1} - \lambda_{\ell}$. Forming $\norow$ does not change these differences for $\ell \leq r-1$, proving part (a). Similarly, forming $\lambda[t]$ from $\lambda$ does not change any differences, proving part (b).
\end{proof}

\begin{thm} \label{thm:PowerSums}
  The boolean number of the Ferrers graph associated to the
  shape $\lambda = (\lambda_1,\ldots,\lambda_r)$ is
  $$\B(\lambda) = \sum_{j=0}^r c_\lambda(r,j) \cdot j^{\lambda_r}.$$
\end{thm}
\begin{proof}
  We prove this by induction on $r$ using the recursive formula in Theorem \ref{thm:Recursion}. The base case, where $r=1$, is easily checked.
  
  To save notation we write $c(i,j)$ instead of $c_\lambda(i,j)$. This
  should not cause any confusion since, for $1 \leq i \leq r-1$ and
  $\ell \leq \lambda_r$, Lemma~\ref{lem:Samec} yields 
 $c_{\lambda}(i,j) = c_{\norow}(i,j) = c_{\norow[-\ell]}(i,j).$
  
  Assuming that the result is true for shapes with at most $r-1$ rows, we have the following sequence of equalities, where the first equality follows from equation~\eqref{eqn:recursive formula}:
  
  \begin{align}
    \B(\lambda) & =  \lambda_r \cdot \B(\norow) + \sum\limits_{\ell=1}^{\lambda_r} \binom{\lambda_r+1}{\ell+1} \cdot \B(\norow[-\ell])\nonumber\\
    & \ = \lambda_r \sum_{j=0}^{r-1} c(r-1,j)\cdot j^{\lambda_{r-1}} + \sum\limits_{\ell=1}^{\lambda_r} \left( \binom{\lambda_r+1}{\ell+1} \cdot \sum_{j=0}^{r-1} c(r-1,j)\cdot j^{\lambda_{r-1}-\ell} \right) \nonumber \\
    & \ = \sum_{j=0}^{r-1}  c(r-1,j) \left( \left( \sum_{\ell=0}^{\lambda_r} \binom{\lambda_r+1}{\ell+1} j^{\lambda_{r-1}-\ell} \right) - j^{\lambda_{r-1}} \right) \nonumber \\
    & \ = \sum_{j=0}^{r-1} c(r-1,j) \cdot \left( \left( j^{\lambda_{r-1}-\lambda_r} \left((j+1)^{\lambda_r+1} -j^{\lambda_r+1} \right)    \right) - j^{\lambda_{r-1}} \right)\nonumber \\
    & \ = \sum_{j=0}^{r-1} c(r-1,j) \cdot \left( j^{\lambda_{r-1}-\lambda_r}(j+1) \left(  \left(j+1\right)^{\lambda_r}-j^{\lambda_r} \right)   \right).\label{eqn:general formula}
  \end{align}
  If we recall that $c(r-1,-1)=c(r-1,r)=0$, then rewriting the sum and
  collecting terms yields
  \[ \sum_{j=0}^{r} \left( (j-1)^{\lambda_{r-1}-\lambda_r}j \cdot c(r-1,j-1) - j^{\lambda_{r-1}-\lambda_r}(j+1) \cdot c(r-1,j) \right) \cdot j^{\lambda_r}  \]
  on the right hand side of \eqref{eqn:general formula}, and as the coefficient of $j^{\lambda_r}$ equals $c(r,j)$, this completes the proof.
\end{proof}

\section{Staircase shapes}\label{staircases}

\begin{defn}
For $r \geq 1$, a \emph{staircase shape of height $r$} is the Ferrers shape $\St_r = (r,r-1,\ldots,2,1)$. 
\end{defn}

For a staircase shape $\St_r$, the recursive formula
for $c_{\St_r}(i,j)$ simplifies to
\[ c_{\St_r}(i,j) = j(j-1) \cdot c_{\St_r}(i-1,j-1) - (j+1)j \cdot c_{\St_r}(i-1,j). \]
Note, in particular, that $c_{\St_r}(i,0) = 0$ for $i>1$.

\begin{table}[htbp] \begin{center}$\begin{array}{c|rrrrrrrr}i & j=0 & 1 & 2 & 3 & 4 & 5 & 6 & 7\\\hline1 & -1 & 1 \\2 & 0 & -2 & 2 \\3 & 0 & 4 & -16 & 12 \\4 & 0 & -8 & 104 & -240 & 144 \\5 & 0 & 16 & -640 & 3504 & -5760 & 2880 \\6 & 0 & -32 & 3872 & -45888 & 157248 & -201600 & 86400 \\7 & 0 & 64 & -23296 & 573888 & -3695616 & 9192960 & -9676800 & 3628800 \\\end{array}$\end{center}\caption{The triangle calculating $c_{\St(7)}(i,j)$.}\end{table}

\begin{cor}\label{cor:L-S}
For $r  \ge 1$, the values
\begin{equation}\label{eqn:L-S}
\left\{\frac{(-1)^{r+j}}{j!(j-1)!} \cdot c_{\St_r}(r,j)\right\}
\end{equation}
are the Legendre-Stirling numbers.
\end{cor}
\begin{proof}
The sequence in \eqref{eqn:L-S} has the same initial values and the same recurrence as the
Legendre-Stirling numbers, given in \cite[A071951]{oeis}.
\end{proof}

The Legendre-Stirling numbers are discussed in \cite{ELW} and \cite{LW}.  From the formula for the Legendre-Stirling numbers (equation \eqref{eq:L-S}), one obtains a formula for $c_{\St_r}(r,j)$, 
and thus for $\B(\St_r)$ as well.

\begin{cor}\label{cor:staircase}
For $r \ge 1$,
$$\B(\St_r) = \sum_{j=1}^r \sum_{\ell=1}^j (-1)^{r+\ell} \frac{(2\ell+1)(\ell^2+\ell)^r \cdot j!j!}{(\ell+j+1)! (j-\ell)!}.$$
\end{cor}

In fact, the boolean numbers of staircase shapes form a known sequence.

\begin{cor}\label{cor:Genocchi}
  The values $\left\{\B(\St_r)\right\}_{r \ge 1}$ are the
  sequence of Genocchi numbers of the second kind.
\end{cor}

\begin{proof}
  The Genocchi numbers of the second kind $\{g(r)\}$ can be calculated by $g(r) =
  G(r,1)$, where $G(r,x)$ is the function defined recursively by
  \begin{align*}
    &G(r,x) = (x+1)^2 G(r-1, x+1)-x(x+1)G(r-1, x)\text{\ \ \ and}\\
    &G(1,x) = 1 
  \end{align*}
  for all $x\geq 0$ and $r\geq 2$.

By a simple induction on $i$, for $1 \le i \le r$, one can prove that $G(r,1)$ equals
 $$\sum_{j=1}^i j \cdot c_{\St_r}(i,j) \cdot G( r+1 -i , j ).$$
 The base case $i=1$ is trivial, and the inductive step
  follows easily from the recursive formulas for $G(r,x)$ and $c_{\St_r}(i,j)$. When $i=r$, we have $G(r+1-i,j) = 1$, and the equation simplifies to
$$g(r) = G(r,1) = \sum_{j=1}^i j \cdot c_{\St_r}(r,j) = \B(\St_r).$$
\end{proof}

Corollaries~\ref{cor:L-S} and~\ref{cor:Genocchi} reveal a relationship
between the Legendre-Stirling numbers $\{d(r,j)\}$ and the Genocchi
numbers of the second kind $\{g(r)\}$ that seems to be new in the
literature.

\begin{cor}
  The Genocchi numbers of the second kind $\{g(r)\}$ and the
  Legendre-Stirling numbers $\{d(r,j)\}$ are related by the equation
  \[ g(r) = \sum\limits_{j=1}^r (-1)^{r+j}(j!)^2 \cdot d(r,j). \]
\end{cor}

Corollary \ref{cor:L-S} says in particular that, for a fixed $j \geq 1$, the sequence
\[ \left\{ \frac{(-1)^{i+j}}{j!(j-1)!} \cdot c_{\sigma_r}(i,j) \right\}_{i\geq 1} \]
has the same generating function as the Legendre-Stirling numbers, namely
\[ \frac{x^j}{\prod\limits_{\ell=1}^{j}\left(1-\ell(\ell+1)x\right)}. \]
This result can be generalized to staircases of other steplengths as follows.

\begin{defn}
For $r, d \geq 1$, a \emph{staircase shape of height $r$ with steplength $d$} is the Ferrers shape $\St_{r,d} = (rd,(r-1)d,\ldots,2d,d)$. 
\end{defn}

\begin{prop}
 Let $\St_{r,d}$ be a staircase shape of height $r$ and steplength $d$, and put 
  \[ \widehat{c}_{\St_{r,d}}(i,j) = \frac{(-1)^{i+j} c_{\St_{r,d}}(i,j)}{j!\left((j-1)!\right)^d}. \]
  For $j \geq 1$, the sequence $\{\widehat{c}_{\St_{r,d}}(i,j) \}_{i\geq 1}$ has generating function
  \[ F_j(x)=\frac{x^j}{\prod\limits_{i=1}^{j}\left(1-i^d(i+1)x\right)}. \]
\end{prop}

\begin{proof}
  The recursive formula for $c_{\St_{r,d}}(i,j)$ is
  \begin{equation}\label{case-d}
    c_{\St_{r,d}}(i,j) 
    = j(j-1)^d \cdot c_{\St_{r,d}}(i-1,j-1) - (j+1)j^d \cdot c_{\St_{r,d}}(i-1,j).
  \end{equation}
  Multiplying equation~(\ref{case-d}) by $\frac{(-1)^{i+j}}{j!((j-1)!)^d}$ yields
  \begin{equation}\label{case-dd}
    \widehat{c}_{\St_{r,d}}(i,j) 
    = \widehat{c}_{\St_{r,d}}(i-1,j-1)+(j+1)j^d\widehat{c}_{\St_{r,d}}(i-1,j).
  \end{equation}
  Let the generating function for the sequence
  $\{\widehat{c}_{\St_{r,d}}(i,j) \}_{i\geq 1}$ be $F_j(x)=\sum_{i\geq 1}\widehat{c}_{\St_{r,d}}(i,j)x^i$.
  Then equation~\eqref{case-dd} gives $F_j(x)=xF_{j-1}(x)+(j+1)j^dxF_j(x)$,
  which shows that
  $$F_j(x)=\frac{x^j}{\prod\limits_{i=1}^{j}\left(1-i^d(i+1)x\right)}.
  $$
\end{proof}

\section{Complete bipartite graphs}\label{complete}

In this section we consider the $r$ row shape
$\lambda=(k,\ldots, k)$. In other words, $\lambda$ is a rectangle
having $r$ rows and $k$ columns. The corresponding Ferrers graph is
the complete bipartite graph $K_{r,k}$. For a rectangular shape, the
recursive formula for $c_{\lambda}(i,j)$ simplifies to
$$c_\lambda(i,j) = j \cdot c_\lambda(i-1,j-1) - (j+1) \cdot c_\lambda(i-1,j).$$

\begin{prop}\label{prop:bipartite}
  For positive integers $r$ and $k$, we have
  $$\B(K_{r,k}) = \sum_{j=1}^{r}(-1)^{r-j}j!{ r+1 \brace j+1 }j^k,$$
  where ${ r+1 \brace j+1 }$ denotes a Stirling number of the
  second kind.
\end{prop}

\begin{proof}
  Let $a(i,j)=(-1)^{i-j}j!{ i+1 \brace j+1 }$. Trivially,
  $a(1,j)=c_{\lambda}(1,j)$. From the familiar recursion ${ i+1 \brace
    j+1 } = { i \brace j }+(j+1) \cdot { i \brace j+1 } $ for the
  Stirling numbers it follows that
  $$a(i,j) = j \cdot a(i-1,j-1) - (j+1) \cdot a(i-1,j).
  $$
  Thus $a(i,j)$ and $c_{\lambda}(i,j)$ satisfy the same recursion.
\end{proof}

Note that this result can also be obtained as a consequence of \cite[Section 5.2]{dpt},
where a similar result is proved for the bivariate chromatic polynomial.

\section{Computational complexity}\label{sec:complexity}

We conclude by addressing the computational complexity (that is, the cost) of calculating
$\B(\lambda)$ by means of a triangle, as in Theorem~\ref{thm:PowerSums}, in order to
illustrate the advantage of this method over calculating by edge recursion. More 
precisely, we show that, while calculating boolean numbers by edge recursion is
$\#P$-hard, calculating the boolean number of a Ferrers shape with $n$ cells using the
results from Section \ref{formula-general} requires $n^2/4 + O(n)$ multiplications. 

The following algorithm calculates the vector $\big( c_\lambda(r,j) \big)_{j=0}^r$.

\medskip

\noindent
{\smaller 1}\quad {\tt def} $\Gamma(\lambda_1,\dots,\lambda_r):$\\
{\smaller 2}\quad\mbox{}\qquad {\tt if} $r=1:$              \\
{\smaller 3}\quad\mbox{}\qquad\qquad {\tt return} $(-1,1)$  \\
{\smaller 4}\quad\mbox{}\qquad {\tt else}\,:                \\
{\smaller 5}\quad\mbox{}\qquad\qquad $(c_{-1}, c_{r}):=(0,0)$  \\
{\smaller 6}\quad\mbox{}\qquad\qquad 
$(c_0,\dots,c_{r-1}) := \Gamma(\lambda_1,\dots,\lambda_{r-1})$ \\
{\smaller 7}\quad\mbox{}\qquad\qquad $d_r\!:=\lambda_{r-1}-\lambda_r$\\
{\smaller 8}\quad\mbox{}\qquad\qquad {\tt return}
$\big(\;\,j(j-1)^{d_r}c_{j-1} -
(j+1)j^{d_r}c_{j} \;\,\big\vert\,\; j = 0,\dots,
r\,\;\big)$\\

We shall analyze the dominant factor in the running time of $\Gamma$,
the number of multiplications that it uses; let $f(\lambda)$ be that
number.  The multiplications are carried out in row 6 (the recursive
call) and row 8. To be precise, for $r>1$ we have
$$f(\lambda_1, \dots, \lambda_r) = f(\lambda_1, \dots, \lambda_{r-1}) 
+ 2(r+1)(d_r+1),$$
and for $r=1$ we have $f((\lambda_1))=0$. Solving this simple
recursion we get
\begin{align*}
  f(\lambda_1, \dots, \lambda_r) 
  = 2\sum_{i=2}^r(i+1)(d_i+1)
  &= 2\sum_{i=2}^rid_i + 2\sum_{i=2}^rd_i + 2\sum_{i=2}^r(i+1) \\
  &= 2\sum_{i=2}^rid_i + 2(\lambda_1-\lambda_r) + r^2+O(r).
\end{align*}

Because $d_i = \lambda_{i-1} - \lambda_i$, we have
$$ \sum_{i=2}^rid_i = \sum_{i=2}^ri\lambda_{i-1} 
     - \sum_{i=2}^r(i+1)\lambda_i + \sum_{i=2}^r\lambda_i = 2\lambda_1 - (r+1)\lambda_r + n - \lambda_1 
  =  \lambda_1 - (r+1)\lambda_r + n,$$
where $n=\lambda_1 +\dots+ \lambda_r$ is the total number of cells.  Thus
\begin{align*}
f(\lambda_1, \dots, \lambda_r) 
&= 2\big(\lambda_1 - (r+1)\lambda_r + n\big) 
+ 2(\lambda_1-\lambda_r) + r^2+O(r)\\
&= 2n + r^2 + O(r\lambda_r + \lambda_1).
\end{align*}

The graphs for $\lambda$ and for its transpose $\lambda'$ have the same
boolean number, so we can assume that $r \leq \lambda_1$. Then $r+\lambda_1 -
1 \le n$, and thus $r \le (n+1)/2$. Also, $r\lambda_r + \lambda_1<2n$. 
Thus,
$$f(\lambda) = 2n + (n+1)^2/4 + O(n) 
= n^2/4 + O(n).$$

\begin{cor}
  The number of multiplications needed to calculate $\B(\lambda)$ using
  Theorem~\ref{thm:PowerSums} is
  $$n^2/4 + O(n),$$
  where $n$ is the total number of
  cells of $\lambda$.
\end{cor}

By comparison, calculating the boolean number using edge recursion is, in
general, $\#P$-hard. This follows from the relationship to the bivariate
chromatic polynomial, $\B(G) = P(G,0,-1)$, and from a result in \cite{H}
saying that computing $P(G,x,y)$ is $\#P$-hard, unless $y = 0$ or 
$(x,y) \in \{ (1,1),(2,2) \}$. 

\begin{rem}
The coefficients $c_{\lambda}(i,j)$ in the triangle for computing $\B(\lambda)$
grow exponentially in $i$ and polynomially in $j$, and the size of these numbers
affects the cost of the algorithm. We obtain an upper bound for the running time,
taking the size of the coefficients into account by making the following 
observations. First, the cost of multiplying two numbers is bounded by 
the product of their binary logarithms. Next, all the multiplications in the
algorithm are of the form $(j+1)j^{d_r}c_j$. We can make the bounds
$\log(j) < \log(j+1) \leq \log(n)$. Also, going through the algorithm, the
order of the size of $\log(c_{j})$ can be bounded by $O(n \log(n))$. Thus the cost
of each multiplication is bounded by $O(n \log(n))$, and the total cost of the
algorithm, accounting for the size of the coefficients, is bounded by $O(n^3 \log(n)^2)$.
\end{rem}

\end{document}